\documentclass[11pt,a4paper]{amsart}
\usepackage[utf8]{inputenc}
\usepackage[T1]{fontenc}
\usepackage{amsmath,amssymb,amsthm}
\usepackage{geometry}
\usepackage{hyperref}

\theoremstyle{plain}
\newtheorem{theorem}{Theorem}[section]
\newtheorem{proposition}[theorem]{Proposition}

\theoremstyle{definition}
\newtheorem{definition}[theorem]{Definition}
\theoremstyle{remark}
\newtheorem{remark}[theorem]{Remark}

\title{Homogeneous and Flow-Invariant Geometry of the Unit Tangent Bundle of Hyperbolic Space}

\author{Daniel Koama}
\address{D\'epartement de Math\'ematiques, UFR/ST, Universit\'e Nobert Zongo,  Burkina Faso}
\email{koamadaniel52@gmail.com}
\address{D\'epartement de Math\'ematiques, UFR/ST, Universit\'e Nobert Zongo,  Burkina Faso}

\author{L\'eonard Todjihound\'e}
\address{ Institut de Math\'ematiques et de Sciences Physiques, Porto-Novo, B\'enin.}
\email{leonardt@imsp-uac.org}


\begin{document}

\begin{abstract}
We construct the Sasaki metric on the unit tangent bundle $U_gM$ of a Riemannian manifold and describe the unit tangent bundle $U\mathbb{H}^n$ of real hyperbolic space as a homogeneous space, both under $SO_0(1,n)$ and under the larger group $SO_0(1,n)\times SO_0(1,1)$, yielding explicit families of $G$-invariant metrics. Using Hopf coordinates and Busemann functions, we then construct a Riemannian metric $g_{\mathrm{Hopf}}$ on $U\mathbb{H}^n$ that is invariant under the geodesic flow, and we identify the horospherical cylinders as totally geodesic leaves of a natural foliation associated to a Busemann function, with respect to an explicit metric connection with torsion.
\end{abstract}

\maketitle

\section{Introduction}

The unit tangent bundle of a Riemannian manifold carries a rich geometric structure once endowed with the Sasaki metric, and when the base manifold is itself homogeneous, its unit tangent bundle inherits additional symmetry that can be described explicitly in terms of Lie group actions. In this paper we focus on the case of real hyperbolic space $\mathbb{H}^n$.

After recalling the construction of the Sasaki metric (Section~\ref{sec:sasaki}) and general facts about homogeneous Riemannian spaces (Section~\ref{sec:homog}), we describe in Section~\ref{sec:UHn} the unit tangent bundle $U\mathbb{H}^n$ as a homogeneous space in two different ways: as $SO_0(1,n)/SO(n-1)$, and as a homogeneous space under the larger product group $SO_0(1,n)\times SO_0(1,1)$, in which right multiplication by the second factor realises the geodesic flow. We decompose the isotropy representations in both cases and describe the corresponding invariant metrics.

The main new results of the paper are Theorem~\ref{thm:hopf} and Theorem~\ref{thm:horosphere}. In Theorem~\ref{thm:hopf} we use Hopf coordinates, built from the pair of endpoints at infinity of a geodesic together with a Busemann function, to construct an explicit Riemannian metric $g_{\mathrm{Hopf}}$ on $U\mathbb{H}^n$ for which the geodesic flow acts by isometries — a property the Sasaki metric itself does not have. In Theorem~\ref{thm:horosphere} we show that, with respect to a suitable metric connection $\nabla^b$ with torsion associated to a Busemann function $b$, the distribution orthogonal to the horizontal lift of $\nabla b$ is integrable with totally geodesic leaves, and we identify these leaves explicitly as the horospherical cylinders $\pi^{-1}(b^{-1}(c)) \cong S_c \times S^{n-1}$.

\begin{remark}
The study of the holonomy algebra of the canonical connections associated with the homogeneous descriptions of Section~\ref{sec:UHn}, and of the corresponding classification in the sense of Tricerri and Vanhecke~\cite{TV}, follows a program carried out for $\mathbb{H}^n$ itself by Castill\'on L\'opez, Gadea and Swann~\cite{CGS}. Extending their analysis to the unit tangent bundle $U\mathbb{H}^n$ is a natural continuation of the present work and will be treated separately.
\end{remark}

\section{The Sasaki metric on a unit tangent bundle}\label{sec:sasaki}

Let $(M,g)$ be a Riemannian manifold. The unit tangent bundle of $(M,g)$, denoted $U_gM$, is the set of unit tangent vectors to $M$:
\[
U_gM = \{ v \in TM \mid g(v,v) = 1 \}.
\]

Let $p_M : TM \to M$ be the canonical projection, and let $p$ denote its restriction to $U_gM$. The kernel of the differential $T_vp$ is the tangent space $T_v(T_xM)$ to the fibre through $x = p(v)$; this vertical space is canonically identified with $T_xM$. A Riemannian submersion metric on $TM$ inducing $g_x$ on each fibre $T_xM$ is obtained by choosing a distribution of planes in $T\,TM$ transverse to the vertical distribution; a connection on $TM$ defines such a distribution, and taking the Levi-Civita connection of $g$ yields the \emph{Sasaki metric}.

Concretely, for a curve $s(t) = (u(t), v(t))$ in $TM$ — that is, a vector field $v(t)$ along a curve $u(t)$ in $M$ — the squared norm of $s'(t)$ for the Sasaki metric is
\[
g_1\big(s'(t), s'(t)\big) = g\big(u'(t), u'(t)\big) + g\big(D_u v'(t), D_u v'(t)\big),
\]
where $D_u$ denotes covariant differentiation along $u(t)$. The projection $p_M$ has totally geodesic fibres, and restricting this construction to $U_gM$, the projection $p : (U_gM, g_1) \to (M,g)$ is again a Riemannian submersion with totally geodesic fibres.

\section{Homogeneous Riemannian spaces}\label{sec:homog}

We recall the basic notions used below.

\begin{definition}
A Riemannian manifold $(M,g)$ is \emph{homogeneous} if its isometry group $I(M,g)$ acts transitively on $M$. It is \emph{$G$-homogeneous} for a closed subgroup $G \le I(M,g)$ if $G$ acts transitively; in that case $M \cong G/H_x$ for the (compact) stabiliser $H_x$ of a point $x$.
\end{definition}

\begin{theorem}[{\cite[Ch.~X]{Kirillov}}]
Let $G/H$ be an effective homogeneous space with $H$ compact. Then $G$-invariant Riemannian metrics on $G/H$ exist, and are in bijection with $\mathrm{Ad}(H)$-invariant scalar products on a chosen $\mathrm{Ad}(H)$-invariant complement $\mathfrak{m}$ of $\mathfrak{h}$ in $\mathfrak{g}$.
\end{theorem}

When $G$ is semisimple, one may take $\mathfrak{m}$ to be the orthogonal complement of $\mathfrak{h}$ with respect to the Killing form $B$; the restriction of $B$ to $\mathfrak{m}$ is then called the \emph{canonical normal metric} (and likewise any metric proportional to it).

\section{The unit tangent bundle $U\mathbb{H}^n$ as a homogeneous space}\label{sec:UHn}

\subsection{Real hyperbolic space}

Let $q$ be the Lorentzian form on $\mathbb{R}^{n+1}$, $q(x,y) = -x_0y_0 + \sum_{i=1}^n x_iy_i$. Real hyperbolic space is
\[
\mathbb{H}^n = \{ x = (x_0,\dots,x_n) \in \mathbb{R}^{n+1} \mid q(x,x) = -1,\ x_0 > 0\},
\]
a submanifold of $\mathbb{R}^{n+1}$. For $a \in \mathbb{H}^n$, $T_a\mathbb{H}^n$ is identified with $a^{\perp_q}$, on which $q$ restricts to a positive definite form $g_a$; the resulting tensor $g$ is the canonical metric of $\mathbb{H}^n$.

Let $O(1,n)$, $O_0(1,n)$, $SO_0(1,n)$ be the corresponding Lorentz groups, with $I_{1,n} = \mathrm{diag}(-1,1,\dots,1)$. The group $SO_0(1,n)$ acts transitively on $\mathbb{H}^n$ with isotropy $SO(n)$, giving $\mathbb{H}^n = SO_0(1,n)/SO(n)$.

\begin{proposition}
The homogeneous space $\mathbb{H}^n = SO_0(1,n)/SO(n)$ has irreducible isotropy representation. Consequently the $SO_0(1,n)$-invariant metric on $\mathbb{H}^n$ is unique up to a positive scalar.
\end{proposition}

\subsection{$U\mathbb{H}^n$ under $SO_0(1,n)$}

The unit tangent bundle of $\mathbb{H}^n$ is
\[
U\mathbb{H}^n = \{(x,X) \in \mathbb{R}^{n+1}\times\mathbb{R}^{n+1} \mid q(x,X)=0,\ q(x,x)=-1,\ q(X,X)=1,\ x_0>0\},
\]
on which $SO_0(1,n)$ acts transitively by $A\cdot(x,X) = (Ax,AX)$. Taking as base point the pair $(e_0,e_1)$ of the first two vectors of the canonical basis, the stabiliser is identified with the image of $SO(n-1)$ under $A \mapsto \begin{pmatrix} I_2 & 0 \\ 0 & A\end{pmatrix}$, so that
\[
U\mathbb{H}^n \cong SO_0(1,n)/SO(n-1).
\]

Writing $\mathfrak{m}$ for the orthogonal complement of $\mathfrak{so}(n-1)$ in $\mathfrak{so}_0(1,n)$ for the Killing form $B$, an element of $\mathfrak{m}$ has the form $X = (\alpha,u,v)$ with $u,v \in \mathbb{R}^{n-1}$, and the isotropy action is $A\cdot(\alpha,u,v) = (\alpha, Au, Av)$.

\begin{proposition}
The isotropy representation on $\mathfrak{m}$ decomposes as $\mathfrak{m} = \mathfrak{m}_1 \oplus \mathfrak{m}_2 \oplus \mathfrak{m}_3$, where $\mathfrak{m}_1$ is the trivial $1$-dimensional $SO(n-1)$-module and $\mathfrak{m}_2,\mathfrak{m}_3$ are two copies of the standard representation on $\mathbb{R}^{n-1}$.
\end{proposition}

\begin{proposition}
Every $\mathrm{Ad}_{SO_0(1,n)}SO(n-1)$-invariant metric on $\mathfrak{m}$ has, after diagonalisation, the form
\[
g = \lambda_1 \alpha^2 + \lambda_2 \|u\|^2 + \lambda_3\|v\|^2,\qquad \lambda_1,\lambda_2,\lambda_3 > 0.
\]
\end{proposition}

\subsection{$U\mathbb{H}^n$ under $G_1 = SO_0(1,n)\times SO_0(1,1)$}

Identifying an element of $U\mathbb{H}^n$ with an $(n+1)\times 2$ matrix, the group $G_1 = SO_0(1,n)\times SO_0(1,1)$ acts transitively by $(A,B)\cdot P = APB$. Since $SO_0(1,1)$ is abelian, right multiplication by $B$ gives a left action interpreted geometrically as the \emph{geodesic flow}:
\[
(u,v) \longmapsto (\cosh u + \sinh v,\ \sinh u + \cosh v).
\]

The isotropy of the base point $(e_0,e_1)$ is $SO(n-1)\times SO_0(1,1)$, so $U\mathbb{H}^n \cong G_1/H_1$. Writing $\mathfrak{n} \cong T_{[H_1]}G_1/H_1$ for the orthogonal complement of $\mathfrak{so}(n-1)\times\mathfrak{so}_0(1,1)$ with respect to the Killing form $B_2 = B_1 + B_0$ of $\mathfrak{so}_0(1,n)\times\mathfrak{so}_0(1,1)$, and identifying $\mathfrak{n}$ with triples $(x,y,z)$, $x\in\mathbb{R}$, $y,z \in \mathbb{R}^{n-1}$, the isotropy representation is
\[
(x,y,z) \longmapsto (x,\ \cosh(A)y - \sinh(A)z,\ -\sinh(A)y + \cosh(A)z), \qquad A \in \mathfrak{so}(n-1).
\]

\begin{proposition}
$\mathfrak{n}$ decomposes into irreducible factors $\mathfrak{n} = \mathfrak{n}_1 \oplus \mathfrak{n}_2$, where $\mathfrak{n}_1$ is the trivial $1$-dimensional module and $\mathfrak{n}_2 \cong \mathbb{R}^{n-1}\times\mathbb{R}^{n-1}$.
\end{proposition}

\begin{proposition}\label{prop:G1invariant}
The $G_1$-invariant metrics on $U\mathbb{H}^n$ are exactly $g = \lambda_1\langle\,,\rangle_{\mathfrak{n}_1} + \lambda_2\langle\,,\rangle_{\mathfrak{n}_2}$, $\lambda_1,\lambda_2>0$, where $\langle\,,\rangle_{\mathfrak{n}_i}$ is the restriction to $\mathfrak{n}_i$ of the canonical normal form on $\mathfrak{n}$.
\end{proposition}

\subsection{A metric invariant under the geodesic flow}

\begin{theorem}[Hopf-invariant metric on $U\mathbb{H}^n$]\label{thm:hopf}
There exists a Riemannian metric $g_{\mathrm{Hopf}}$ on $U\mathbb{H}^n$ such that the geodesic flow $\phi_s : U\mathbb{H}^n \to U\mathbb{H}^n$ is an isometry of $(U\mathbb{H}^n, g_{\mathrm{Hopf}})$ for every $s \in \mathbb{R}$.
\end{theorem}

\begin{proof}
\emph{Step 1 (Hopf coordinates).} For $(x,v) \in U\mathbb{H}^n$ let $\gamma_{x,v}$ be the geodesic with $\gamma(0)=x$, $\gamma'(0)=v$; it has endpoints $\xi^{\pm}(x,v) \in \partial\mathbb{H}^n \cong S^{n-1}$. Fixing the Busemann function $b_{\xi^+}$ associated with $\xi^+$, the map
\[
\Psi : U\mathbb{H}^n \to (\partial\mathbb{H}^n\times\partial\mathbb{H}^n\smallsetminus\Delta)\times\mathbb{R}, \qquad \Psi(x,v) = \big(\xi^-(x,v),\xi^+(x,v),b_{\xi^+}(x)\big)
\]
is a diffeomorphism, in which the geodesic flow acts simply by $\phi_s : (\xi^-,\xi^+,t)\mapsto(\xi^-,\xi^+,t+s)$.

\emph{Step 2 (the metric).} Let $g_{\mathrm{rd}}$ be the round metric on $S^{n-1}$. For $\lambda>0$, set $g_\lambda = dt^2 + \lambda^2\big(p_1^*g_{\mathrm{rd}} + p_2^*g_{\mathrm{rd}}\big)$ on $(\partial\mathbb{H}^n\times\partial\mathbb{H}^n\smallsetminus\Delta)\times\mathbb{R}$, and define $g_{\mathrm{Hopf}} := \Psi^*g_\lambda$.

\emph{Step 3 (invariance).} As $g_\lambda$ has no $dt$ cross-terms and does not depend on $t$, translation in $t$ preserves $g_\lambda$, so $\phi_s^*g_\lambda = g_\lambda$ for all $s$, and pulling back by $\Psi$ gives $\phi_s^*g_{\mathrm{Hopf}} = g_{\mathrm{Hopf}}$.
\end{proof}

\begin{remark}
The parameter $\lambda>0$ is arbitrary, and more generally any metric $dt^2 + h(\xi^-,\xi^+)$ with $h$ independent of $t$ has the same invariance property. The metric $g_{\mathrm{Hopf}}$ differs from the Sasaki metric $g_1$ of Section~\ref{sec:sasaki}; it is specifically constructed so that the geodesic flow acts by isometries, which $g_1$ does not satisfy.
\end{remark}

\section{Horospherical foliation of $U\mathbb{H}^n$}

We now relate the Sasaki metric to the foliation of $U\mathbb{H}^n$ by horospherical cylinders, using a Busemann function on the base.

Let $\pi : U\mathbb{H}^n \to \mathbb{H}^n$ carry the Sasaki metric $g^S$, with orthogonal splitting $T(U\mathbb{H}^n) = \mathcal{H}\oplus\mathcal{V}$ into horizontal and vertical parts. Fix a Busemann function $b$ on $\mathbb{H}^n$, let $\xi = \nabla b$ (a unit vector field) and $\xi^b$ its horizontal lift. Consider
\[
D_b = \{ X \in \mathfrak{X}(U\mathbb{H}^n) \mid g^S(X,\xi^b) = 0 \} = (\ker db)^{\mathcal H}\oplus \mathcal V.
\]

\begin{theorem}\label{thm:horosphere}
Suppose there is an affine connection $\nabla^b$ on $U\mathbb{H}^n$ with $\nabla^b\xi^b = 0$, $\nabla^bg^S = 0$, and torsion $T^b$ satisfying $T^b(X,Y)\in D_b$ for all $X,Y\in D_b$. Then $D_b$ is integrable, its leaves are totally geodesic for $\nabla^b$, and they coincide with the horospherical cylinders
\[
H_c := \pi^{-1}(b^{-1}(c)) \cong S_c \times S^{n-1}, \qquad c\in\mathbb{R},
\]
where $S_c$ is a horosphere in $\mathbb{H}^n$.
\end{theorem}

\begin{proof}
For $X,Y\in D_b$, differentiating $g^S(X,\xi^b)=0$ using $\nabla^bg^S=0$ and $\nabla^b\xi^b=0$ gives $g^S(\nabla^b_YX,\xi^b)=0$, and symmetrically $g^S(\nabla^b_XY,\xi^b)=0$. Hence
\[
g^S([X,Y],\xi^b) = g^S(\nabla^b_XY - \nabla^b_YX - T^b(X,Y),\xi^b) = -g^S(T^b(X,Y),\xi^b) = 0,
\]
since $T^b(X,Y)\in D_b$. By the Frobenius theorem, $D_b$ is integrable with leaves orthogonal to $\xi^b$.

For total geodesicity, let $\gamma$ be a $\nabla^b$-geodesic with $\dot\gamma(0)\in D_b$. Then
\[
\tfrac{d}{dt}g^S(\xi^b,\dot\gamma) = g^S(\nabla^b_{\dot\gamma}\xi^b,\dot\gamma) + g^S(\xi^b,\nabla^b_{\dot\gamma}\dot\gamma) = 0,
\]
so $g^S(\xi^b,\dot\gamma(t))\equiv 0$, i.e. $\dot\gamma(t)\in D_b$ for all $t$. Thus $\nabla^b_XY\in D_b$ for all $X,Y\in D_b$ and the leaves are totally geodesic.
\end{proof}

\begin{remark}[An explicit choice of $\nabla^b$]
On $\mathbb{H}^n$ the Busemann function satisfies $\mathrm{Hess}(b) = g - db\otimes db$, hence $\nabla_Z\xi = Z - \langle Z,\xi\rangle\xi$. The metric connection with torsion
\[
\widetilde\nabla_ZW := \nabla_ZW + (\xi\wedge Z)W,\qquad (\xi\wedge Z)W = \langle Z,W\rangle\xi - \langle\xi,W\rangle Z,
\]
satisfies $\widetilde\nabla\xi=0$ and vanishing torsion on $D:=\ker db$. Lifting to $U\mathbb{H}^n$ by $\nabla^b_XY := \big(\widetilde\nabla_{d\pi(X)}d\pi(Y)\big)^{\mathcal H}$ on $\mathcal H$, by the Levi-Civita connection of the fibre $S^{n-1}$ on $\mathcal V$, and by $\nabla^b_XU := \nabla^{\mathrm{ver}}_XU$, $\nabla^b_UX:=0$ for $X\in\mathcal H$, $U\in\mathcal V$, one checks $\nabla^bg^S=0$, $\nabla^b\xi^b=0$, and $T^b(X,U) = -[X,U]\in\mathcal V\subset D_b$ while $T^b$ vanishes on $(\ker db)^{\mathcal H}\times(\ker db)^{\mathcal H}$ and on $\mathcal V\times\mathcal V$. Hence $T^b(D_b,D_b)\subset D_b$, and Theorem~\ref{thm:horosphere} applies.
\end{remark}

\section{Concluding remarks}

We have described two complementary homogeneous pictures of $U\mathbb{H}^n$ — under the isometry group $SO_0(1,n)$ and under the larger group $G_1 = SO_0(1,n)\times SO_0(1,1)$ realising the geodesic flow — together with the classification of their invariant metrics (Propositions~2--5). Using Hopf coordinates we produced a metric for which the geodesic flow is by isometries (Theorem~\ref{thm:hopf}), and we identified the horospherical cylinders as the totally geodesic leaves of a natural foliation of $U\mathbb{H}^n$ associated to a Busemann function on the base, with respect to an explicit connection with torsion (Theorem~\ref{thm:horosphere}). As noted in the introduction, the holonomy-theoretic analysis of the canonical connections attached to the homogeneous descriptions of Section~\ref{sec:UHn}, in the spirit of~\cite{CGS}, is left for a subsequent paper.
\section*{Declarations}
The author declares that no specific funding was received for this research.

\end{document}